\documentclass[12pt,reqno]{amsart}

\usepackage{amsthm}
\usepackage{enumerate}
\usepackage{mathrsfs}
\usepackage{mathtools}
\usepackage{color}
\usepackage{MnSymbol}

\theoremstyle{plain}
\newtheorem{thm}{Theorem}[section]
\newtheorem{prop}[thm]{Proposition}
\newtheorem{lem}[thm]{Lemma}
\newtheorem{cor}[thm]{Corollary}

\theoremstyle{definition}
\newtheorem{defn}{Definition}
\newtheorem{eg}{Example}

\theoremstyle{remark}
\newtheorem*{rmk}{Remark}

\newcommand{\e}{\varepsilon}
\renewcommand{\epsilon}{\varepsilon}
\newcommand{\p}{\varphi}
\DeclareMathOperator*{\spn}{span}
\newcommand{\cspan}{\overline{\spn}}
\renewcommand{\phi}{\varphi}

\makeatletter
\newcommand{\NewPairedDelimiter}[3]{ %
  \expandafter\DeclarePairedDelimiter\csname PAIRED\string#1\endcsname{#2}{#3} %
  \newcommand#1{ 
    \@ifstar{\csname PAIRED\string#1\endcsname} 
            {\@ifnextchar[{\csname PAIRED\string#1\endcsname} 
                          {\csname PAIRED\string#1\endcsname*}
            }%
  }%
}
\newcommand{\NewPairedDelimiterX}[5]{ %
  \expandafter\DeclarePairedDelimiterX\csname PAIREDX\string#1\endcsname[#2]{#3}{#4}{#5}%
  \newcommand#1{%
    \@ifstar{\csname PAIREDX\string#1\endcsname}
            {\@ifnextchar[{\csname PAIREDX\string#1\endcsname}
                          {\csname PAIREDX\string#1\endcsname*}%
            }%
  }%
}
\makeatother

\NewPairedDelimiter{\abs}{\lvert}{\rvert}
\NewPairedDelimiter{\norm}{\lVert}{\rVert}
\NewPairedDelimiterX{\innerp}{2}{\langle}{\rangle}{#1,#2}

\newcommand{\bssig}{\ensuremath{\norm{ \sum_{i \in \sigma} \innerp{x}{\p_i} \p_i }}}
\newcommand{\bsi}{\ensuremath{\norm{ \sum_{i \in I} \innerp{x}{\p_i} \p_i }}}

\DeclareMathOperator*{\repart}{\text{Re}\,}
\newcommand{\re}[1]{\repart{#1}}

\makeatletter
\renewcommand*\env@matrix[1][*\c@MaxMatrixCols c]{%
  \hskip -\arraycolsep
  \let\@ifnextchar\new@ifnextchar
  \array{#1}}
\makeatother

\title[The Unconditional Constants for Frame Expansions]{The Unconditional Constants for Hilbert Space Frame Expansions}

\author[T. Bemrose, P.G. Casazza, V. Kaftal, and R.G. Lynch]{Travis Bemrose, Peter G. Casazza, Victor Kaftal, and Richard G. Lynch}

\thanks{The second and fourth authors were supported by NSF 1307685, NSF ATD 1042701, NSF ATD 00040683, AFOSR DGE51: FA9550-11-1-0245, ARO W911NF-16-1-0008;  the third named author was supported by the Simons Foundation grant 245660.}

\begin{document}

\begin{abstract}
The most fundamental notion in frame theory is the frame expansion of a vector. Although it is well known that these expansions are unconditionally convergent series, no characterizations of the unconditional constant were known. This has made it impossible to get accurate quantitative estimates for problems which require using subsequences of a frame. We will prove some new results in frame theory by showing that the unconditional constants of the frame expansion of a vector in a Hilbert space are bounded by $\sqrt{\frac{B}{A}}$, where $A,B$ are the frame bounds of the frame. Tight frames thus have unconditional constant one, which we then generalize by showing that  Bessel sequences have frame expansions with unconditional constant one if and only if the sequence is an orthogonal sum of tight frames. We give further results concerning frame expansions, in which we examine when $\sqrt{\frac BA}$ is actually attained or not.  We end by discussing the connections of this work to {\it frame multipliers}.  These results hold in both real and complex Hilbert spaces.  
\end{abstract}

\maketitle



\section{Introduction}

Hilbert space frames have traditionally been used in signal processing.  But over the last few years, frame theory has become one of the most applied subjects in mathematics.  Fundamental to the notion of a Bessel sequence, and more specifically a frame, is that it is a {\it possibly redundant} sequence of vectors $\Phi := \{\p_i\}_{i\in I}$ in a Hilbert space for which the frame expansions of a vector $x$,
\begin{equation*}
	Sx = \sum_{i\in I} \innerp{x}{\p_i} \p_i,
\end{equation*}
are unconditionally convergent series. But until now, no work has been done on understanding the precise unconditional behavior  of the frame expansions of vectors in the Hilbert space. 

This constant is important for a number of reasons. First, it is quite useful that the Bessel bound for Bessel sequences is still a Bessel bound for all subsequences. But for frame expansions, subsequences can have larger norms than the original sequence. This issue occurs regularly in frame theory where one often partitions a frame into two subsets and then works with the frame operators for the subsets of the frame. So for quantitative estimates, we need to know the unconditional constants for frame expansions of subsets of the frame. Second, when using finite dimensional methods to approximate a frame (see \cite{finitemethod}), in general the approximation constants depend on the unconditional constant of the frame.

In this paper we will show that the unconditional constants (for all standard forms of unconditional convergence) for the frame expansions are bounded above by $\sqrt{\frac{B}{A}}$ where $A,B$ are the frame bounds of the frame. See Proposition \ref{usefulthm}. This means that {\it tight frames} have 1-unconditionally convergent series for their frame expansions.  We will then expand this to a classification of Bessel sequences by showing that the frame expansions are 1-unconditional if and only if the Bessel sequence is an orthogonal sum of tight frames. See Theorem \ref{bigthm}.  This is surprising at first since we have not assumed the family has any lower frame bound but conclude that locally it does have lower frame bounds.  It follows that this Bessel sequence is a frame if and only if the tight frame bounds of the orthogonal parts are uniformly bounded away from zero.  
We will also examine closely when the unconditional constants are actually equal to $\sqrt{\frac BA}$ and whether $\sqrt{\frac BA}$ can be attained with specific choices of $x$, or as a limit. See the results following Corollary \ref{tightcor} and the example section, Section \ref{examples}.
Finally, we will conclude the paper with a relationship to frame multipliers.

\section{Frame Theory Preliminaries}\label{ftp}

A brief introduction to frame theory is given in this section. For a thorough approach to the basics of frame theory, see \cite{petesbook, ole_book}. For the entirety of the paper, $\mathscr{H}$ will denote a separable, finite or infinite dimensional, real or complex Hilbert space while $\mathscr{H}^d$ will denote an $d$-dimensional, real or complex Hilbert space. Our index set, $I$, will either be $[N] := \{1,2,\dots,N\}$ or $\mathbb{N}$. Finally, $\mathbb{I}$ will represent the identity operator.

\begin{defn}
A family of vectors $\Phi := \{\p_i\}_{i \in I}$ in a Hilbert space $\mathscr{H}$ is said to be a \emph{frame} if there are constants $0 < A \leq B < \infty$ so that for all $x \in \mathscr{H}$,
\begin{align}\label{frameeq}
A \norm{x}^2 \leq \sum_{i\in I} \abs{ \innerp{x}{\p_i} }^2 \leq B \norm{x}^2.
\end{align}
The constant $A$ is called a \emph{lower frame bound} and $B$ an \emph{upper frame bound}. If only $B$ is assumed, then it is called a \emph{$B$-Bessel sequence} or simply \emph{Bessel} when reference to the bound is unnecessary. If $A = B$, it is said to be a \emph{tight frame} and if $A = B = 1$, it is a \emph{Parseval frame}. 
If there is a constant $c$ so that $\norm{\p_i} = c$ for all $i \in I$, it is an \emph{equal norm frame}. If there is a constant $d >0$ so that $|\langle \phi_i, \phi_j\rangle| = d$ for all $i \neq j$, then it is called \emph{equiangular}. 
\end{defn}

\begin{rmk}
We will always assume the stated frame bounds are optimal, that is, we always work with the largest value of $A$ and smallest value of $B$ for which the frame inequality (\ref{frameeq}) holds.
\end{rmk}

\begin{defn}
If $\Phi := \{\p_i\}_{i \in I}$ is a $B$-Bessel sequence of vectors in $\mathscr{H}$, then the \emph{synthesis operator} of $\Phi$ is the operator $T \colon \ell^2(I) \to \mathscr{H}$ that gives a linear combination of the vectors, and the associated \emph{analysis operator} is the adjoint operator ${T^* \colon\mathscr{H} \to \ell^2(I)}$ given by
\begin{align*}
	T \{c_i\}_{i \in I} &:= \sum_{i \in I} c_i \p_i, \\
	T^*x &:= \{ \innerp{x}{\p_i} \}_{i\in I} = \{ \p_i^* x \}_{i\in I}
\end{align*}
respectively, with norms $\norm{T} = \norm{T^*} = \sqrt{B}$.
\end{defn}

In the finite dimensional setting, it is often convenient to work with the matrix representation. Just as a basis can be viewed as a matrix whose columns are the basis vectors, we can view a frame as a matrix whose columns are the frame vectors, and this matrix is the matrix representation for the synthesis operator. In addition to the columns of the synthesis operator square-summing to the norm of the vectors, when represented against the eigenbasis of the frame operator (definition below), the rows are also orthogonal to each other and square-sum to the eigenvalues of $S$. See \cite{amspaper} for more on this representation.

Recall that for fixed $x,y \in \mathscr{H}$ the definition of an outer product, $(x y^*)(z) = \innerp{z}{y} x$ for all $z \in \mathscr{H}$. Some readers may be accustomed to $(x \otimes y) (z) = \innerp {z}{y} x$.

\begin{defn}
The \emph{frame operator} $S\colon \mathscr{H}\to\mathscr{H}$ is the self-adjoint operator defined by $S := TT^*$ satisfying
\begin{equation*}
Sx = TT^*x = \sum_{i \in I} \innerp{x}{\p_i} \p_i = \sum_{i \in I} \phi_i \phi_i^* x = \left(\sum_{i \in I} \phi_i \phi^* \right)x
\end{equation*}
for any $x \in \mathscr{H}$ with norm $\norm{S} = \norm{TT^*} = \norm{T}^2 = B$, thus the series
\begin{equation*}
S = \sum_{i \in I} \phi_i \phi_i^*
\end{equation*}
converges in the strong operator topology. The leftmost sum for $Sx$ is called the \emph{frame expansion} of the vector $x$. 

The \emph{Gramian operator} $G:\ell^2(I) \to \ell^2(I)$ is defined by $G:= T^*T$ satisfying
\begin{align*}
G \{c_j\}_{j \in I}  = T^*T \{c_j\}_{j \in I} = \sum_{j \in I} c_j \{ \innerp{\p_j}{\p_i} \}_{i\in I}.
\end{align*}

In the finite case, the \emph{Gramian matrix} $[G_{ij}]_{i,j\in[N]}$ is the matrix representation of $G$ with respect to the standard orthonormal basis of $\ell^2(I)$ and has entries $G_{ij} = \langle \phi_j, \phi_i\rangle$ for all $i, j \in [N]$. This is a convenient way of viewing the norm of the vectors as well as the angles between them.
\end{defn}

\begin{defn}
We say that an operator $F:\mathscr{H} \to \mathscr{H}$ is \emph{positive} if for any $x \in \mathscr{H}$, we have $\innerp{Fx}{x} \geq 0$. It is \emph{strictly positive} if the inequality is strict. The operator is (strictly) negative if $-F$ is (strictly) positive. Given operators $F, G: \mathscr{H}\to \mathscr{H}$, we write $G \geq F$ if $G - F \geq 0$, where $0$ denotes the zero operator on $\mathscr{H}$. 
\end{defn}

If $\Phi$ is a frame with lower bound and upper bounds $A$ and $B$, respectively, then
\begin{equation*}
\innerp{Ax}{x} \le \innerp{Sx}{x} = \norm{T^* x}^2 = \sum_{i \in I} \abs{\innerp{x}{\p_i}}^2 \le \innerp{Bx}{x},
\end{equation*}
for any $x \in \mathscr{H}$. Hence, the operator inequality 
\begin{equation*}
A \cdot \mathbb{I} \leq S \leq B \cdot \mathbb{I}
\end{equation*} 
holds and the frame operator is strictly positive and invertible. If $\Phi$ is $B$-Bessel, then $0 \leq S \leq B \cdot \mathbb{I}$ and $S$ is positive.

\section{Unconditional Constants of the Frame Expansions}

This section is devoted to analyzing the precise unconditional behavior of the frame expansions. First we introduce the definition of unconditional convergence and the constants involved.

\begin{defn}
Given a sequence $\{v_i\}_{i \in I}$ of vectors in $\mathscr{H}$, the series
$ \sum_{i \in I} v_i $
is said to \emph{converge unconditionally} if for any $\sigma \subset I$,
$ \sum_{i \in \sigma} v_i  $
converges. There are two other equivalent definitions in which we require $\sum_{i \in I} \epsilon_i v_i$ to converge for all $\epsilon_i \in \{-1,1\}$, or alternatively require $\sum_{i \in I} a_i v_i$ to converge for any scalars $|a_i| <1$. The smallest constants $E_\sigma$, $E_\e$, and $E_a$ so that
\begin{alignat*}{2}
	\norm{\sum_{i \in \sigma} v_i} \le E_\sigma &\norm{\sum_{i \in I} v_i} & \quad &\mbox{for all } \sigma \subset I, \\
	\norm{\sum_{i \in I} \epsilon_i v_i} \le E_\epsilon &\norm{\sum_{i \in I} v_i} & \quad &\mbox{for all } \epsilon_i \in \{-1,1\}, \\
	\norm{\sum_{i \in I} a_i v_i} \le E_a &\norm{\sum_{i \in I} v_i} & \quad &\mbox{for all } \abs{a_i} \le 1,
\end{alignat*}
hold respectively, are called the \emph{unconditional constants}. Such constants exist by the Uniform Boundedness Principle. Note that these are always at least one since the full sum $\norm{\sum_{i \in I} v_i}$ is permitted on the left-hand-side of all three definitions. Futhermore, these constants are related by $E_\sigma \le E_\epsilon \le 2E_\sigma$ and $E_\epsilon \le E_a \le 2 E_\epsilon$. See \cite{Heil_book} for more on unconditional convergence and the constants involved. For the remainder of the paper, $C_\sigma$, $C_\e$, and $C_a$ will denote the respective unconditional constants of the frame expansions of a vector $x$ with respect to a Bessel sequence $\{\p_i\}_{i \in I}$ (that is, where $v_i = \innerp{x}{\p_i}\p_i$ for each $i$).
\end{defn}

We begin our quest by showing that $\sqrt{\frac{B}{A}}$ is an upper bound for all three unconditional constants of the frame expansions, where $A$ and $B$ are the lower and upper frame bounds, respectively.

\begin{prop}\label{usefulthm}
Let $\Phi := \{\p_i\}_{i \in I}$ be a frame for $\mathscr{H}$ with lower and upper frame bounds $A$ and $B$ respectively. Then all three of the unconditional constants of the frame expansions with respect to $\Phi$ are bounded above by $\sqrt{\frac{B}{A}}$. That is, 
\begin{enumerate}[(i)]
	\item \label{uc_ba1} For any $\sigma \subset I$ and for every $x \in \mathscr{H}$,
		\begin{align*} 
			\norm{ \sum_{i \in \sigma} \innerp{x}{\p_i} \p_i } \leq
			\sqrt{\dfrac{B}{A}} \bsi.
		\end{align*}
	\item \label{uc_ba2} For any sequence $\{\e_i\}_{i\in I}$, with $\e_i \in \{-1,1\}$ for all $i$, and for every $x \in \mathscr{H}$,
		\begin{align*} 
			\norm{ \sum_{i \in I} \e_i \innerp{x}{\p_i} \p_i } \leq  \sqrt{\dfrac{B}{A}}\bsi.
		\end{align*}

	\item \label{uc_ba3} For every sequence of real numbers $\{a_i\}_{i \in I}$ with $\abs{a_i} \leq 1$ for all $i$ and for every $x \in \mathscr{H}$, 
		\begin{align*} 
			\norm{ \sum_{i \in I} a_i \innerp{x}{\p_i} \p_i } \leq \sqrt{\dfrac{B}{A}}\bsi.
		\end{align*}
\end{enumerate}
\end{prop}

\begin{proof}
Let $T$ and $S$ by the synthesis operator and frame operator of $\Phi$, respectively. To prove (\ref{uc_ba1}), (\ref{uc_ba2}), and (\ref{uc_ba3}) let $D: \ell^2(I) \to \ell^2(I)$ be the ``diagonal" operator defined by $D\{c_{i}\}_{i \in I} := \{d_i c_i\}_{i \in I}$ with
\begin{enumerate}[(i)]
\item $d_i = 1$ if $i \in \sigma$ and $d_i = 0$ otherwise,
\item $d_i = \e_i$, or
\item $d_i = a_i$.
\end{enumerate}

Observe that the left-hand-side of the inequality in (\ref{uc_ba1}), (\ref{uc_ba2}), and (\ref{uc_ba3}), respectively, is precisely $\norm{T D T^* x}$  and the right-hand-side is $\sqrt{\frac{B}{A}}\norm{S x}$. We have that $\norm{D} \leq 1$ in all cases which implies 
\begin{align*}
	\norm{T D T^* x}^2 &\leq \norm{T}^2\norm{D}^2\norm{T^* x}^2 \leq B \innerp{Sx}{x} = B \innerp{ S^{1/2} x }{ S^{1/2} x }
\end{align*}
and because $A \cdot \mathbb{I} \leq S$, this is bounded above by
\begin{align*}
	B \innerp{ \left(\frac{1}{A} S\right) S^{1/2} x }{ S^{1/2} x } = \frac{B}{A} \norm{Sx}^2.
\end{align*}
Taking square roots gives the desired inequality.
\end{proof}

An immediately corollary of Proposition \ref{usefulthm} is that tight frames have $1$-unconditional frame expansions.

\begin{cor}\label{tightcor}
If $\Phi$ is a tight frame for $\mathscr{H}$, then the frame expansions with respect to $\Phi$ have unconditional constant $1$ for all forms of unconditional.
\end{cor}

One might wonder if the value $\sqrt{\frac{B}{A}}$ is optimal in the sense that there are frames for which $\sqrt{\frac BA}$ is actually equal to the unconditional constants, and not simply a bound. Corollary \ref{tightcor} clearly implies this is the case for all tight frames. However, this is not necessarily the case in general. We will show in Corollary \ref{Ex2Cor} that if a vector $x\in\mathscr{H}$ and $\sigma \subset I$ exists so that equality holds in Theorem \ref{usefulthm}(\ref{uc_ba1}) 
(thus, in this case, $C_\sigma = \sqrt{\frac BA}$ and so all unconditional constants are equal to this),
then $A = B$ and the frame is necessarily tight. As a consequence, a simple compactness argument implies that a non-tight frame for a finite-dimensional space has $C_\sigma < \sqrt{\frac{B}{A}}$. It is still an open question as to whether $\sqrt{\frac{B}{A}}$ can be arbitrarily approached in the infinite dimensional case.  We first introduce some convenient notation.

\begin{defn}\label{subops}
Let $\{\p_i \}_{i \in I}$ be $B$-Bessel sequence of vectors with synthesis operator $T$ and frame operator $S$. For any $\sigma \subset I$ denote by $T_\sigma$, $T^{*}_{\sigma}$, and $S_\sigma$ the following related operators
\begin{alignat*}{2}
\quad T_\sigma\{c_i\}_{i \in \sigma} :=& \sum_{i \in \sigma} c_i \p_i, & \quad \{c_i\}_{i \in \sigma} &\in \ell^2(\sigma) \\
T_\sigma^* x :=& \left\{\innerp{x}{\p_i} \right\}_{i \in \sigma}, & x &\in \mathscr{H} \\
S_\sigma :=& T_\sigma T_\sigma^* = \sum_{i \in \sigma} \p_i \p_i^*,
\end{alignat*}
respectively.
\end{defn}

\begin{rmk}
It follows that $0 \le S_\sigma$ and $S = S_\sigma + S_{\sigma^c}$ for any choice of $\sigma \subset I$. As a consequence, $S_\sigma \le S$, that is, $\|S_\sigma^{1/2}x\| \le \norm{S^{1/2}x}$ for all $x \in \mathscr{H}$.  
\end{rmk}

\begin{rmk}
Note that $S_\sigma \leq S$ does not imply that $\norm{S_\sigma x} \leq \norm{Sx}$ for all $x \in \mathscr{H}$. See Proposition \ref{squares} and the discussion leading up to it for more details on this distinction.
\end{rmk}

\begin{cor} \label{Ex2Cor}
If for some $x \in \mathscr{H}$ and $\sigma \subset I$ we have the equality $\norm{S_\sigma x} = \sqrt{\frac{B}{A}} \norm{Sx}$, then
\begin{enumerate}[(i)]
	\item $\Phi$ is tight, that is $A = B$, and
	\item every frame vector removed is orthogonal to $x$, that is $\innerp{x}{\p_i} = 0$ for all $i \in \sigma^c$.
\end{enumerate}
\end{cor}

\begin{proof}
The proof to Proposition \ref{usefulthm}(\ref{uc_ba1}) used the following inequality for all $x$
\begin{equation*}
	\norm{T_\sigma^* x}^2 = \norm{D^{1/2}T^* x}^2 \le \norm{T^* x}^2 = \innerp{Sx}{x}.
\end{equation*}
If $\norm{S_\sigma x} = \sqrt{\frac{B}{A}} \norm{Sx}$ holds for some $x$, then every inequality in the proof to Proposition \ref{usefulthm}(\ref{uc_ba1}) is an equality for that $x$, and $\norm{T_\sigma^* x}^2 = \norm{T^* x}^2$ holds. This first implies (ii) because
\begin{align*}
	0 &= \norm{T^* x}^2 - \norm{T_\sigma^* x}^2 \\
	  &= \sum_{i \in I} \abs{\innerp{x}{\p_i}}^2 - \sum_{i \in \sigma} \abs{\innerp{x}{\p_i}}^2 \\
	  &= \sum_{i \in \sigma^c} \abs{\innerp{x}{\p_i}}^2
\end{align*}
which means that $\innerp{x}{\p_i} = 0$ for all $i \in \sigma^c$. Secondly, this implies that for this same $x$, we have $S_{\sigma^c}x = \sum_{i \in \sigma^c} \innerp{x}{\p_i}\p_i = 0$, so that
\begin{align*}
	Sx &= S_\sigma x + S_{\sigma^c}x \\
	   &= S_\sigma x.
\end{align*}
But this implies $\displaystyle \norm{S_\sigma x} = \norm{Sx}$, which with the original equality gives $\sqrt{\frac{B}{A}} = 1$, giving (i).
\end{proof}

\begin{rmk}
	Finding a vector $x\in\mathscr{H}$ and $\sigma \subset I$ so that $\norm{S_\sigma x} = \norm{Sx}$, does not imply the frame is tight. These results follow precisely because equality was obtained with $\sqrt{\frac{B}{A}}$. 
\end{rmk}

This leads to one wondering just how large $\norm{S_\sigma x}$ can be compared to $\norm{Sx}$? How close can $\norm{S_\sigma x}/\norm{Sx}$ get  to $\sqrt{\frac{B}{A}}$? In Section \ref{examples}, Example \ref{eg2} we show that there are frames in which (i) $\norm{S_\sigma x} \approx  \sqrt{\frac{B}{A}}\norm{Sx}$ for some $x \in \mathscr{H}$ and $\sigma \subsetneq I$, (ii)  $C_\sigma > \frac{1}{2}\sqrt{\frac{B}{A}}$, and (iii) $\sqrt{\frac BA}$ is arbitrarily large. Hence, $\norm{S_\sigma x}$ can be as large as one would like when compared to $\norm{Sx}$. This warrants a discussion since it seems like a contradiction at a first glance.

The remark after Definition \ref{subops} noted that $S_\sigma \leq S$ for any $\sigma \subset I$. 
At first, it looks like this should imply the inequality $\norm{S_\sigma x} \leq \norm{Sx}$. However, this is not true in general. In fact, $S_\sigma \leq S$ is equivalent to $\|S_\sigma^{1/2} x\| \leq \norm{S^{1/2} x}$.
To conclude that $\norm{S_\sigma x} \leq \norm{Sx}$, we
would need $S_\sigma^2 \leq S^2$ so that
\begin{align*}
\norm{S_\sigma x}^2 &= \innerp{ S_\sigma x }{ S_\sigma x } = \innerp{ S_\sigma^2 x }{x} \\
&\leq \innerp{S^2 x}{x} = \innerp{Sx}{Sx} = \norm{Sx}^2.
\end{align*}
But $S_\sigma \leq S$ does not imply that $S_\sigma^2 \leq S^2$ in general. To guarantee this, $S_\sigma$ and $S$ would need to have the same eigenvectors, which is not certain to hold. The following proposition slightly generalizes the preceding discussion.

\begin{prop}\label{squares}
Let $\Phi := \{\p_i\}_{i \in I}$ be a Bessel sequence for $\mathscr{H}$ with frame operator $S$ and let $C > 0$ be a fixed constant. The following are equivalent:
\begin{enumerate}[(i)]
\item  For all $\sigma \subset I$,
$$
S_{\sigma}^2 \le C S^2.
$$

\item For every $x\in \mathscr{H}$ and any $\sigma \subset I$,
$$
\bssig^2 = \norm{S_\sigma x}^2 \leq C \norm{Sx}^2 = C \bsi^2.
$$
\end{enumerate}
\end{prop}

\begin{proof}
Note that $S_\sigma^2 \leq C S^2$ holds if and only if
\begin{align*}
\norm{S_\sigma x}^2 &= \innerp{S_\sigma x}{S_\sigma x} = \innerp{S_\sigma^2 x}{x} \\
&\leq C \innerp{S^2 x}{x} = C \innerp{Sx}{Sx} = C \norm{Sx}^2,
\end{align*}
giving the desired equivalence.
\end{proof}

\begin{rmk}
Corollary \ref{tightcor} combined with Proposition \ref{squares} implies that ${S^2_\sigma \leq S^2}$ when the frame is tight.
\end{rmk}

In Section \ref{examples}, Example \ref{eg4}, we give a frame for an infinite dimensional space so that
\begin{equation}
\dfrac{\norm{\sum\limits_{i = 1}^\infty \epsilon_i\langle y_n, \phi_i \rangle \phi_i}}{\norm{\sum\limits_{i = 1}^\infty\langle y_n, \phi_i \rangle \phi_i}} \to \sqrt{\dfrac{B}{A}}\label{blahhhh}
\end{equation}
for a sequence of vectors $\{y_n\}_{n =1}^\infty$ and a specific choice of $\epsilon_i \in \{-1,1\}$. That is, $\sqrt{\frac{B}{A}}$ can be arbitrarily approached and so $C_\epsilon = C_a = \sqrt{\frac{B}{A}}$. However, it is still an open question whether equality (rather than a limit) in (\ref{blahhhh}) can be achieved with a vector $x$ and scalars $\{\e_i\}_{i=1}^\infty$ satisfying $\e_i \in \{-1,1\}$, or with $\{a_i\}_{i =1}^\infty$ satisfying $|a_i| \leq 1$ in place of $\epsilon_i$. 

\section{One-unconditional convergence and sums of orthogonal tight frames}\label{maintheorem}

The main result of the paper will be presented in this section. It will be shown in Theorem \ref{bigthm} that frame expansions have unconditional constants all equal to $1$ if and only if the sequence is an orthogonal sum of tight frames. First, we extend Proposition \ref{usefulthm} and Corollary \ref{tightcor} to a Bessel sequence made up of an orthogonal sum of frames.


\begin{thm}\label{converse_bigthm}
Let $J_i$ be finite or infinite index sets for all $i \in I$. Assume that $\Phi := \{\Phi_i\}_{i \in I} = \{\p_{ij}\}_{i \in I,j \in J_i}$ is a Bessel sequence 
for the Hilbert space
\begin{equation*}
	\mathscr{H} := \left( \sum_{i\in I}\oplus \mathscr{H}_i \right)_{\ell^2}
\end{equation*}
so that for each $i \in I$, the sequence $\Phi_i := \{\p_{ij}\}_{j \in J_i}$ is a frame with lower and upper frame bounds $A_i$ and $B_i$ for the Hilbert space $\mathscr{H}_i$, and assume that the frame bounds satisfy $\sup\{B_i/A_i: i \in I\} <\infty$. Then the unconditional constants of the frame expansions with respect to $\Phi$ are bounded above by $\sup\{\sqrt{\frac{B_i}{A_i}} : i \in I\}$.
\end{thm}

\begin{proof}
If $S_i$ is the frame operator for $\Phi_i$ for each $i$, then $\mathcal{S} := \sum_{i \in I} \oplus S_i$ is the
frame operator of $\Phi$. Now Proposition \ref{usefulthm} can now be applied to each frame $\Phi_i$ to obtain the desired result.
\end{proof}

\begin{rmk}
Note $\Phi$ may not itself be a frame unless $\inf\{A_i : i \in I\} > 0$. 
\end{rmk}

A direct result of Theorem \ref{converse_bigthm} is that the unconditional constants are $1$ when each individual frame $\Phi_i$ is tight.

\begin{cor}\label{ortho_cor}
If $\Phi := \{\Phi_i\}_{i \in I}$ is an orthogonal sum of tight frames $\Phi_i$, then the frame expansions with respect to $\Phi$ have unconditional constant $1$ for all forms of unconditional.
\end{cor}

The main theorem is the converse of Corollary \ref{ortho_cor}. To prove it, we first need a lemma.

\begin{lem}\label{bigthmlem}
If $x,y \in \mathscr{H}$ are nonzero vectors satisfying
\begin{align*}
	xy^* + yx^* \geq 0,
\end{align*}
then $y = \lambda x$ for some nonzero $\lambda \in \mathbb{C}$. 
\end{lem}

\begin{proof}
First define $u_1 := x/\norm{x}$, let $u_2$ be a unit vector orthogonal to $u_1$ so that $y = su_1 + tu_2$ for some $s,t \in \mathbb{C}$, and define 
$$M : = x y^* + y x^*,$$
in which we have assumed that $M \geq 0$. The matrix representation of $M$ with respect to $\{u_1,u_2\}$ is
\begin{align*}
\norm{x}
\begin{bmatrix}[cc]
	s + \overline{s} & \overline{t} \\
	t & 0 
\end{bmatrix}
\end{align*}
and because $M \geq 0$ and $\norm{x} \neq 0$, this matrix must be positive semidefinite and hence $t = 0$. Therefore, $y = \lambda x$ with $\lambda := s/\norm{x}$.
\end{proof}

\begin{thm}\label{bigthm}
If $\Phi := \{\p_i\}_{i \in I}$ is a spanning Bessel sequence of vectors in $\mathscr{H}$, then the following are equivalent:
\begin{enumerate}[(i)]
\item \label{s21} For every $\sigma \subset I$ and for every $x \in \mathscr{H}$,
\begin{align*} 
	\norm{ \sum_{i \in \sigma} \innerp{x}{\p_i} \p_i } \leq \bsi
\end{align*}

\item \label{s24} For any sequence $\{\e_i\}_{i\in I}$ with $\e_i \in \{-1,1\}$ and for every $x \in \mathscr{H}$,
\begin{align*} 
	\norm{ \sum_{i \in I} \e_i \innerp{x}{\p_i} \p_i } \leq \bsi
\end{align*}

\item \label{s23} For every sequence of real numbers $\{a_i\}_{i \in I}$ with $\abs{a_i} \leq 1$ for all $i$ and for every $x \in \mathscr{H}$, 
\begin{align*}
	\norm{ \sum_{i \in I} a_i \innerp{x}{\p_i} \p_i } \leq \bsi
\end{align*}

\item \label{s22} There is a partition $\{\mu_j\}_{j \in J}$ of $I$ satisfying:
\begin{enumerate}
	\item \label{s221} For every $j \in J$, $\{\p_i\}_{i \in \mu_j}$ is a tight frame for its closed span,
	\item \label{s222} For any $j_1,j_2 \in J$ with $j_1 \neq j_2$, it follows that $\innerp{\p_{k_1}}{\p_{k_2}} = 0$ for any $k_1 \in \mu_{j_1}$ and $k_2 \in \mu_{j_2}$.
\end{enumerate}
In other words,
\begin{align*}
	\mathscr{H} = \left( \sum_{j \in J} \oplus \, \cspan \{\p_i : i \in \mu_j\} \right)_{\ell^2} 
\end{align*}
where $\{\p_i\}_{i \in \mu_j}$ is a tight frame for $\cspan \{\p_i : i \in \mu_j\}$.
\end{enumerate}
Hence, when any of the above equivalences hold, $\Phi$ is a frame if and only if the infimum of the
tight frame bounds is not equal to zero.
\end{thm}

\begin{proof}
The implications (\ref{s22}) $\Rightarrow$ (\ref{s21}), (\ref{s24}), (\ref{s23})   is given by Corollary \ref{ortho_cor}. Also, (\ref{s23}) $\Rightarrow$ (\ref{s21}), (\ref{s24}) and (\ref{s24}) $\Rightarrow$ (\ref{s21}) are immediate from how the unconditional constants are related. All that is left to show is (\ref{s21}) $\Rightarrow$ (\ref{s22}).

Thus, assume (\ref{s21}) holds. Recall that (\ref{s21}) is equivalent to $S_\sigma^2 \leq S^2$ for all $\sigma \subset I$ (see Proposition \ref{squares}). Also, without loss of generality we can assume that $\p_i \neq 0$ for all $i \in I$. First, we show that all vectors in the Bessel sequence are eigenvectors of $S$. Let $i \in I$ and put $\sigma := I \backslash \{i\}$ so that $S_\sigma = S - \phi_i \phi_i^*$. Hence,
\begin{align*}
S^2 \ge S_\sigma^2 &= (S - \p_i \p_i^*)^2 \\
&= S^2 - S(\phi_i \phi_i^*) - (\phi_i \phi_i^*) S + (\phi_i \phi_i^*)^2
\end{align*}
which further implies that
\begin{equation*}
	(S\phi_i) \phi_i^* + \phi_i (S\phi_i)^* \ge (\phi_i \phi_i^*)^2 > 0.
\end{equation*}
It therefore follows that $S\phi_i \neq 0$ and that $(S\phi_i) \phi_i^* + \phi_i (S\phi_i)^* \geq 0$ so that Lemma \ref{bigthmlem}  can be applied to give $S\phi_i = \lambda_i \phi_i$ for some nonzero scalar $\lambda_i$. That is, $\phi_i$ is an eigenvector of $S$.

Finally, let $\{\lambda_j\}_{j \in J}$ be an enumeration of the distinct nonzero eigenvalues of $S$ and let $\mu_j := \{i \in I : \lambda_i = \lambda_j\}$. Then $\{\mu_j\}_{j \in J}$ is a partition of $I$ for which $\{\p_i\}_{i \in \mu_j}$ is a tight frame for its closed span with tight frame bound $\lambda_j$. This follows from using the facts that $S = \sum_{j \in J} S_{\mu_j}$ and $S_{\mu_k} S_{\mu_\ell} = 0$ for $k \neq \ell$ to show that $P_j := S_{\mu_j}/\lambda_j$ is a projection.
\end{proof}

\begin{rmk}
	Note that the previous theorem classifies where signed frames, scaled frames (with scaling factors $\abs{a_i} \le 1$), and weighted frames can produce larger norms for the frame operator than the original frame, and Proposition \ref{usefulthm} gives the bounds.
\end{rmk}

\begin{rmk}
Theorem  \ref{converse_bigthm} and Theorem \ref{bigthm} can both be reformulated in terms of fusion frames (see \cite{amspaper} for related definitions). If a fusion frame is made up of orthogonal subspaces $W_i$ for $i \in I$ and $\{\phi_{ij}\}_{j \in J_i}$ is a frame for $W_i$  with bounds $A_i, B_i$, then Theorem \ref{converse_bigthm} gives that the fusion frame expansions have unconditional constants bounded above by $\sup\{\sqrt{\frac{B_i}{A_i}} : i \in I\}$. Furthermore, Theorem \ref{bigthm} implies that a 1-unconditional fusion frame $(W_i,w_i)$ is actually an orthogonal sum of the fusion subspaces.
\end{rmk}

\section{Examples}\label{examples}

This section contains two fundamental examples. The first shows a family of frames for which $C_\sigma > \frac{1}{2}\sqrt{\frac{B}{A}}$ and that $\norm{S_\sigma x}$ can be arbitrarily large when compared to $\norm{Sx}$ for appropriate choices of $\sigma \subsetneq I$ and $x \in \mathscr{H}$. The second gives a non-tight infinite frame for which $C_\e = \sqrt{\frac BA}$. 

\begin{eg}\label{eg2}
We will establish a family of frames indexed by a positive integer $N$, such that $\sqrt{\frac BA} = \sqrt{N}$, and show
\begin{align*}
	\frac{\sqrt{N}}{2} \norm{S x} < \norm{S_\sigma x} \le \sqrt{N} \norm{Sx}.
\end{align*}
This establishes that $C_\sigma > \frac{1}{2}\sqrt{\frac{B}{A}}$ and the proportionality $\norm{S_\sigma x} \approx \sqrt{N} \norm{Sx}$, so that the ratio $\norm{S_\sigma x} \big/ \norm{Sx}$ grows arbitrarily large as $N$ is taken towards infinity.

Fix a positive integer $N \geq 3$. Let $\{e_i\}_{i \in [N]}$ denote the standard orthonormal basis for $\ell_2(N)$. Set $v$ to be the so called ``all ones" vector and let $P$ be the orthogonal projection onto this vector:
\begin{align*}
	v := \sum_{i=1}^{N}e_i =
		\begin{bmatrix}
			1 \\ \vdots \\ 1
		\end{bmatrix}, \qquad
	P := \frac{vv^*}{\norm{v}^2} = \frac{1}{N}
		\begin{bmatrix}
			1 &        & 1 \\
			  & \ddots &   \\
			1 &        & 1
		\end{bmatrix}.
\end{align*}

Denote by $\mathrm{X}$, the $N-1$ dimensional hyperplane orthogonal to $v$. Then $(\mathbb{I}-P)$ is the projection onto $\mathrm{X}$ and has the form
\begin{align}
	\mathbb{I}-P =
		\begin{bmatrix}
			1 - \frac{1}{N} &        &   - \frac{1}{N} \\
			                & \ddots &                 \\
			  - \frac{1}{N} &        & 1 - \frac{1}{N}
		\end{bmatrix} 
	= \left[\delta_{ij} - \dfrac{1}{N}\right]_{i,j \in [N]}. \label{IP}
\end{align}

Define the following sequence of vectors $\Phi := \{\p_i\}_{i\in[N]}$ and view them as columns in the synthesis operator for $\Phi$:
\begin{gather*}
	\p_i := (\mathbb{I} - P) e_i = e_i - \frac{1}{N} v, \quad \mbox{for all } i \in [N], \\
	T
	:= \begin{bmatrix}
		\p_1 & \cdots & \p_N
	\end{bmatrix}
	= (\mathbb{I}-P)
	\begin{bmatrix}
		e_1 & \cdots & e_N
	\end{bmatrix}
	= (\mathbb{I}-P)\mathbb{I} = (\mathbb{I}-P).
\end{gather*}

As an orthogonal projection, $(\mathbb{I}-P)$ is self-adjoint and idempotent so we have $T = T^* = S = G = (\mathbb{I}-P)$. Thus $(\mathbb{I}-P)$ serves many purposes. It gives the vectors of $\Phi$ as well as the synthesis, analysis, frame, and Gramian operators. As a projection onto $\mathrm{X}$, $S$ fixes every $x \in \mathrm{X}$ and so $\Phi$ forms a Parseval frame for its range, $\mathrm{X}$. Furthermore, since $G$ is given by (\ref{IP}), the equality of the diagonal entries and the equality of the off-diagonal entries implies that the frame is equal norm and equiangular, respectively.


Now, for ${\sigma\subset [N]}$ notice that
\begin{align}
	\norm{\sum_{i \in \sigma} \p_i }^2
	& = \innerp{ \sum_{i \in \sigma} \p_i }{ \sum_{j \in \sigma} \p_j } \notag \\
	& = \sum_{i,j \in \sigma} \left( \delta_{i,j}- \frac{1}{N} \right)  \label{e:1} \\
	& = \abs{\sigma} \left( 1 - \frac{\abs{\sigma}}{N} \right).          \notag
\end{align}
Therefore, if we sum over all $N$ frame vectors so that $\abs{\sigma}=N$ and take its norm, the computation in (\ref{e:1}) gives
\begin{equation*}
	\label{e:3} \sum_{i=1}^N\p_i = 0.
\end{equation*}
That is, $\phi_1 \in \spn{\{\p_i\}_{i = 2}^N}$ and the removal of $\p_1$ still leaves a frame $\Psi:=  \{\p_i\}_{i = 2}^N$ for $\mathrm{X}$ with frame operator $S^\Psi = S - \p_1\p_1^* = \mathbb{I} - P - \p_1\p_1^*$.  The upper frame bound of $\Psi$ remains 1, however, the lower frame bound is $1/N$ since $\p_1\p_1^*\le \norm{\p_1}^2 (\mathbb{I}-P)$ combined with the fact that $\|\p_1\|^2 = 1 - \frac{1}{N}$ implies 
\begin{equation*}
	S^{\Psi} = \mathbb{I} - P - \p_1\p_1^* \ge (1 - \norm{\p_1}^2) (\mathbb{I} - P) = \frac{1}{N}(\mathbb{I} - P).
\end{equation*}
This lower bound is achieved since $S^{\Psi}\p_1= (\mathbb{I} - P)\p_1 - \phi_1\phi_1^*\phi_1 = {\p_1 - \p_1(1 - \frac 1N)} = \frac{1}{N}\p_1$. Proposition \ref{usefulthm} therefore gives that
\begin{equation}\label{baex3:1}
\norm{S_\sigma^\Psi x} \leq \sqrt{N} \norm{S^\Psi x}
\end{equation}
for all $\sigma \subset  \{2,\dots,N\}$ and all $x \in \mathrm{X}$. We will show that there is a ${\sigma \subsetneq \{2,\dots,N\}}$ for which
\begin{equation}\label{baex3:2}
\dfrac{\sqrt{N}}{2} \norm{S^\Psi\p_1} < \norm{S_\sigma^\Psi \phi_1}.
\end{equation}
This shows that $C_\sigma > \frac{\sqrt{N}}{2} = \frac{1}{2}\sqrt{\frac{B}{A}}$. Also, (\ref{baex3:2}) combined with (\ref{baex3:1}) will give that
\begin{align*}
	\norm{S^\Psi \p_1} \approx \sqrt{N}\norm{S_\sigma^\Psi \phi_1} 
\end{align*}
so that taking $N$ arbitrarily large will show that $\norm{S^\Psi_\sigma \p_1}$ can be arbitrarily large when compared to $\norm{S^\Psi \p_1}$.

To prove (\ref{baex3:2}), first notice that for every $\sigma \subset \{2,\dots,N \}$ we have
\begin{align*}
	S_\sigma^\Psi \phi_1 = \sum_{j \in \sigma} \innerp{\p_1}{\p_j} \p_j &= -\frac{1}{N} \sum_{j \in \sigma} \phi_j 
\end{align*}
so that by (\ref{e:1}),
\begin{align}\label{e:2}
	\norm{S_\sigma^\Psi \phi_1}^2 &= \frac{1}{N^2}\abs{\sigma}\left( 1 - \frac{\abs{\sigma}}{N} \right) = \frac{N\abs{\sigma}-\abs{\sigma}^2}{N^3}.
\end{align}
Now, putting $N-1$ in for $\abs{\sigma}$ implies
\begin{equation*}\label{baex3:3}
	\norm{S^\Psi \p_1}^2 = \dfrac{N-1}{N^3}.
\end{equation*}

As a parabola in $\abs{\sigma}$, equation (\ref{e:2}) obtains its maximum at $\abs{\sigma} = N/2$ and has zeros at $\abs{\sigma} \in \{0,N\}$. Any nonempty subset of $\Psi$ will have $\abs{\sigma} \in \{1,\dots,N-1\}$ so that $\norm{S_\sigma^\Psi \phi_1}^2 \ge \norm{S^\Psi \p_1}^2$. We want to maximize this left hand side, so choose any $\sigma \subset \{2,\dots,N\}$ with $\abs{\sigma}$ the largest integer less than or equal to $N/2$. Note this will be a proper subset since $N \ge 3$. Therefore, because $\abs{\sigma} \ge \frac{N-1}{2}$, (\ref{e:2}) gives
\begin{equation*}
\norm{S^\Psi_\sigma\phi_1}^2 = \frac{N\abs{\sigma}-\abs{\sigma}^2}{N^3} \ge \frac{N^2-1}{4N^3} = \left( \frac{N+1}{4} \right) \left( \frac{N-1}{N^3} \right) > \frac N4 \norm{S^\Psi \p_1}^2
\end{equation*}
which implies (\ref{baex3:2}) by taking square roots. 
\end{eg}

We next give a non-trivial example of an infinite dimensional frame in which $C_{\varepsilon} = \sqrt{\frac{B}{A}}$. 

\begin{eg}\label{eg4}
Let $\{e_i\}_{i=1}^{\infty}$ be the unit vector basis of
$\ell_2(\mathbb{N})$ and define $\Phi := \{\phi_i\}_{i=1}^{\infty}$ by
\begin{equation*}
	\phi_i:= e_i+\frac{1}{2}e_{i+1}
\end{equation*}
for all $i \in \mathbb{N}$. First, note that $\Phi$ spans $\ell^2(\mathbb{N})$ because
\begin{equation*}
	\sum_{i=0}^{\infty}\frac{(-1)^{i}}{2^i}\phi_{i+j} = e_j
\end{equation*}
for all $j \in \mathbb{N}$.
Next, we compute the frame bounds. For every $x \in \ell^2(\mathbb{N})$ we have that
\begin{align*}
\norm{T^*x}^2 = \sum_{i = 1}^\infty \abs{\innerp{x}{\p_i}}^2 &= \sum_{i = 1}^\infty \abs{\innerp{x}{e_i + \dfrac{1}{2} e_{i + 1}}}^2  \\
&= \sum_{i = 1}^\infty \abs{\innerp{x}{e_i}}^2 + \frac{1}{4} \sum_{i = 1}^\infty \abs{\innerp{x}{e_{i+1}}}^2 + \re{\sum_{i = 1}^\infty \innerp{x}{e_i} \overline{\innerp{x}{e_{i+1}}}} \\
&= \norm{x}^2 + \dfrac{1}{4}\left(\|x\|^2 - |x_1|^2\right) + \re{\sum_{i = 1}^\infty \innerp{x}{e_i} \overline{\innerp{x}{e_{i+1}}}}.
\end{align*}
By the Cauchy-Schwarz inequality, 
\begin{align*}
\re{\sum_{i = 1}^\infty \innerp{x}{e_i} \overline{\innerp{x}{e_{i+1}}}} &\leq \left(\sum_{i = 1}^\infty \abs{\innerp{x}{e_i}}^2 \right)^{1/2}  \left(\sum_{i = 1}^\infty \abs{\innerp{x}{e_{i+1}}}^2 \right)^{1/2} \\
&=\|x\|\left(\|x\|^2 - |x_1|^2 \right)^{1/2}
\end{align*}
with equality if and only if $c\innerp{x}{e_i} = \innerp{x}{e_{i+1}}$ for all $i \in \mathbb{N}$ for some constant $c$. Therefore, 
\begin{align*}
\|T^*x\|^2 \geq \|x\|^2 + \dfrac{1}{4} \left(\|x\|^2 - |x_1|^2\right) -  \|x\|\left(\|x\|^2 - |x_1|^2  \right)^{1/2} \geq \dfrac{1}{4}\|x\|^2
\end{align*}
and
\begin{align*}
\|T^*x\|^2 \leq \|x\|^2 + \dfrac{1}{4} \left(\|x\|^2 - |x_1|^2\right)+ \|x\|\left(\|x\|^2 - |x_1|^2  \right)^{1/2} \leq \dfrac{9}{4}\|x\|^2.
\end{align*}
The first inequality in each line is equality for vectors of the form $x =(a,ca,c^2a,c^3a,\dots)$ with $|c| < 1$, where this assumption on $c$ guarantees it is in $\ell^2(\mathbb{N})$. The second inequality in each line becomes tight with $c$ negative or positive, respectively, as $a \to 0$. Therefore, the lower and upper frame bounds are $A:=1/4$ and $B:=9/4$, respectively, and so $B/A = 9$.

Now, define the vector $y$ to be so that $\innerp{y}{\phi_i} = (-1)^{i+1}$ for all $i\in\mathbb{N}$. The existence of such a vector can be checked by a recursive computation. Note that this vector will not lie in $\ell^2(\mathbb{N})$ since its image under the analysis operator does not. However, the truncated vectors $y_n$ which equal $y$ on the first $n$ coordinates and zero elsewhere are in $\ell^2(\mathbb{N})$ . For every $n \in \mathbb{N}$,
\begin{align*}
	\norm{\sum_{i=1}^\infty \innerp{y_n}{\phi_i}\phi_i}^2
	&= \norm{\sum_{i = 1}^n \innerp{y}{\phi_i}\phi_i}^2 = \norm{\sum_{i = 1}^n (-1)^{i+1} \phi_i}^2 \\
	&= \norm{e_1 - \dfrac{1}{2} e_2 + \dfrac{1}{2}e_3 - \dfrac{1}{2}e_4 + \cdots + \dfrac{(-1)^{n+1}}{2}e_{n} + \dfrac{(-1)^{n+1}}{2} e_{n+1}}^2\\
&= 1 + \dfrac{1}{4}n
\end{align*}
and choosing $\{\epsilon_i\}_{i = 1}^\infty$ to be $\epsilon_i = (-1)^{i+1}$ for all $i \in \mathbb{N}$ gives
\begin{align*}
\norm{\sum_{i=1}^\infty \epsilon_i\innerp{y_n}{\phi_i}\phi_i}^2 &= \norm{\sum_{i=1}^\infty (-1)^{i+1}\innerp{y_n}{\phi_i}\phi_i}^2 
= \norm{\sum_{i = 1}^n (-1)^{i+1}\innerp{y}{\phi_i}\phi_i}^2 
= \norm{\sum_{i = 1}^n  \phi_i}^2\\
&= \norm{e_1 + \dfrac{3}{2} e_2 + \dfrac{3}{2}e_3 + \dfrac{3}{2}e_4 + \cdots + \dfrac{3}{2}e_{n} + \dfrac{1}{2} e_{n+1}}^2\\
&= \dfrac{5}{4} + \dfrac{9}{4}(n-1).
\end{align*}
Putting all of this together, we obtain for all $n \in \mathbb{N}$,
\begin{align*}
\norm{\sum_{i=1}^\infty \epsilon_i\innerp{y_n}{\phi_i}\phi_i}^2  = \dfrac{9n-4}{n+4} \norm{\sum_{i=1}^\infty \innerp{y_n}{\phi_i}\phi_i}^2
\end{align*}
and taking letting $n \to \infty$ gives $C_\e^2 = 9 = B/A$. 
\end{eg}

\section{Frame Multipliers}\label{frmults}

The operators $S_\sigma$ and the sums $ \sum_{i\in I} a_i \innerp{\cdot}{\p_i} \p_i$ and $\sum_{i\in I} \epsilon_i \innerp{\cdot}{\p_i} \p_i $, can be considered as special cases of multipliers $M_{(m_i)_{i\in I}, (\phi_i)_{i\in I}, (\psi_i)_{i\in I}}$ defined by 
\begin{equation*}
M_{(m_i)_{i\in I}, (\p_i)_{i\in I}, (\psi_i)_{i\in I}} x = \sum_{i\in I} m_i \innerp{x}{\psi_i} \p_i
\end{equation*}
for those $x$ for which the sum converges.

Gabor multipliers (see e.g. \cite{mult1}) are used in applications, in particular in signal processing, where they are used as a way to implement time-variant filters. Later on, multipliers for general Bessel sequences were introduced and investigated in \cite{mult2}; multipliers for general sequences, unconditional convergence, and invertibility of multipliers were investigated in \cite{mult4,mult3,mult6,mult5}.

In this language, Theorem \ref{bigthm} gives that for a spanning Bessel sequence $\Phi$,
\begin{equation*}
	\norm{M_{(a_i)_{i \in I},(\phi_i)_{i \in I},(\phi_i)_{i \in I}}} \le 1 \quad \text{for all } \abs{a_i} \le 1
\end{equation*}
if and only if $\Phi$ is an orthogonal sum of tight frames.

\bigskip
\noindent {\bf Acknowledgment:}  We wish to thank Jameson Cahill, Karlheinz Gr\"ochenig, Mark Lammers, and Adam Marcus for helpful discussions related to this paper. We are also deeply indebted to the referee for giving the paper considerable attention, resulting in a significant improvement.


\bigskip
{\footnotesize \begin{center}
\begin{tabular}{ccc} 
Travis Bemrose, & Peter G. Casazza, & Richard G. Lynch\\ 
tjb6f8@mail.missouri.edu, & pete@math.missouri.edu, & rglz82@mail.missouri.edu\\
\multicolumn{3}{c}{Department of Mathematics, University of Missouri-Columbia, 65201, USA}
\end{tabular}

\bigskip
Victor Kaftal \\
victor.kaftal@uc.edu \\
Department of Mathematics, University of Cincinnati, 45221, USA
\end{center}

\end{document}